\documentclass[11pt, reqno]{amsart}
\usepackage{amsmath, amsfonts, amsthm, amssymb, multicol, mathtools, dsfont, mathrsfs}
\usepackage{graphicx}
\usepackage{float, hyperref}
\usepackage{scalerel,stackengine, subcaption}
\usepackage[usenames,dvipsnames,x11names]{xcolor}
\usepackage{enumitem}
\usepackage{pgfplots}
\usepgfplotslibrary{fillbetween}
\pgfplotsset{width=10cm,compat=1.9}
\usepackage[square,sort,comma,numbers]{natbib}
\makeatletter
\g@addto@macro\bfseries{\boldmath}
\makeatother

\makeatletter
\def\@setauthors{%
  \begingroup
  \def\thanks{\protect\thanks@warning}%
  \trivlist
  \centering\footnotesize \@topsep30\p@\relax
  \advance\@topsep by -\baselineskip
  \item\relax
  \author@andify\authors
  \def\\{\protect\linebreak}

  \normalsize\lowercase{\authors}%
  
	\ifx\@empty\contribs
  \else
    ,\penalty-3 \space \@setcontribs
    \@closetoccontribs
  \fi
  \endtrivlist
  \endgroup
}
\def\@settitle{\begin{center}
\LARGE\lowercase{\@title}
  \end{center}%
}
\makeatother
\newcommand{\authoremail}[1]{\email{\href{mailto:#1}{\color{lightblue}{#1}}}}
\newcommand{\authoraddress}[1]{\address{\normalfont{#1}}}

\numberwithin{equation}{section}

\newtheorem{thm}{Theorem}[section]

\newtheorem{cor}[thm]{Corollary}

\newtheorem{prop}[thm]{Proposition}

\newtheorem{ques}[thm]{Question}

\renewcommand{\epsilon}{\varepsilon}

\newcommand{\rd}{\mathbb{R}^d}

\renewcommand{\geq}{\geqslant}
\renewcommand{\leq}{\leqslant}

\newcommand{\hd}{\dim_{\textup{H}}}
\newcommand{\frd}{\dim_{\textup{Fr}}}

\newcommand{\fs}{\dim^\theta_{\mathrm{F}}}
\newcommand{\fd}{\dim_{\mathrm{F}}}
\newcommand{\sd}{\dim_{\mathrm{S}}}

\newcommand{\R}{\mathbb{R}}

\renewcommand{\Re}{\text{Re\,}}
\renewcommand{\Im}{\text{Im\,}}

\newcommand{\J}{\mathcal{J}}

\makeatletter
\DeclareRobustCommand\widecheck[1]{{\mathpalette\@widecheck{#1}}}
\def\@widecheck#1#2{%
    \setbox\z@\hbox{\m@th$#1#2$}%
    \setbox\tw@\hbox{\m@th$#1%
       \widehat{%
          \vrule\@width\z@\@height\ht\z@
          \vrule\@height\z@\@width\wd\z@}$}%
    \dp\tw@-\ht\z@
    \@tempdima\ht\z@ \advance\@tempdima2\ht\tw@ \divide\@tempdima\thr@@
    \setbox\tw@\hbox{%
       \raise\@tempdima\hbox{\scalebox{1}[-1]{\lower\@tempdima\box
\tw@}}}%
    {\ooalign{\box\tw@ \cr \box\z@}}}
\makeatother

\stackMath
\newcommand\reallywidehat[1]{%
\savestack{\tmpbox}{\stretchto{%
  \scaleto{%
    \scalerel*[\widthof{\ensuremath{#1}}]{\kern.1pt\mathchar"0362\kern.1pt}%
    {\rule{0ex}{\textheight}}
  }{\textheight}%
}{2.4ex}}%
\stackon[-6.9pt]{#1}{\tmpbox}%
}
\usepackage{fancyhdr}
\definecolor{lightblue}{HTML}{2B77A4}
\colorlet{plotblue}{LightSkyBlue3!80}
\definecolor{darkred}{HTML}{9E0D0D}
\hypersetup{
	colorlinks=true,
	linkcolor=darkred,
	urlcolor=darkred,
	citecolor=lightblue
}
\newcommand\numberthis{\addtocounter{equation}{1}\tag{\theequation}}

\title[Fourier restriction for the additive Brownian sheet]{Fourier restriction for the additive Brownian sheet}

\author{Jonathan M. Fraser}
\authoraddress{J. M. Fraser, University of St Andrews, Scotland}
\authoremail{jmf32@st-andrews.ac.uk}
\thanks{JMF was financially supported by an EPSRC Open Fellowship (EP/Z533440/1) and a Leverhulme Trust Research Project Grant (RPG-2023-281)}
\author{Ana E. de Orellana}
\authoraddress{A. E. de Orellana, University of St Andrews, Scotland}
\authoremail{aedo1@st-andrews.ac.uk}
\thanks{AEdO was financially supported by the University of St Andrews.}

\begin{document}

\thispagestyle{empty}
\begin{abstract}
The Fourier restriction  problem asks when it is meaningful to restrict the Fourier transform of a function to a given set. Many of the key examples are smooth co-dimension  1 manifolds, although there is increasing interest in fractal sets.  Here we propose a natural intermediary problem where one considers the fractal surface generated by the graph of the additive Brownian sheet in $\mathbb{R}^k$.  We obtain the first non-trivial estimates in this direction, giving both a sufficient condition on the range of $q\in[1,2]$ for the Fourier transform to be $L^{q}(\R^{k+1})\to L^2(G(W))$ bounded and a necessary condition for it to be $L^{q}(\R^{k+1})\to L^p(G(W))$ bounded.  The sufficient condition  is obtained via the Fourier spectrum, which is a family of dimensions that interpolate between the Fourier and Hausdorff dimensions. Our main technical result, which is of interest in its own right, gives a precise formula for the Fourier spectrum of the natural measure on the graph of the additive Brownian sheet, and we apply this result to the Fourier restriction problem. Our restriction estimate  is stronger than the estimate obtained from the well-known Stein--Tomas restriction theorem for all $k\geq3$. We obtain the necessary condition in two different ways, one via the Fourier spectrum and one via an appropriate  Knapp   example.   
\end{abstract}
\maketitle
\tableofcontents

\section{Introduction}

\subsection{Fourier restriction and fractal surfaces}


The  Fourier restriction problem asks when it is meaningful to restrict the Fourier transform of a function to a set of measure zero.  Particular instances of this problem have become some of the most famous and long-standing problems in harmonic analysis, going back to work by Stein (see \cite{Ste70, Tom75}), and have deep connections with  geometric measure theory and  PDEs. Some of the most important examples include smooth co-dimension 1 manifolds such as the sphere, cone and paraboloid. Although Fourier restriction is usually cast in the context of smooth manifolds, there is a significant amount of work for the case where the set is of fractal dimension. In the fractal setting, it was first studied in \cite{Moc00}, and saw continuous development in the subsequent years \cite{Mit02,BS11,Che14}.  One of the main goals of this article is to motivate the study of the restriction problem for a natural family of (topological) co-dimension 1 manifolds which are not smooth but rather are realised as the graph of Brownian motion or more generally the Brownian sheet. One may consider the graphs of Brownian motion or the Brownian sheet to be the canonical examples of a fractal surface. In some sense, these examples sit in-between the fractal and smooth settings and exhibit many interesting geometric and analytical features, making them good candidates for Fourier restriction and the related problem of deriving the Fourier spectrum.   Indeed, such graphs  are currently attracting a lot of attention in the context of Fourier dimension, see \cite{FS18,  LL25+, LT25+}.

We now describe the Fourier restriction problem and the (additive) Brownian sheet  in more detail. Given a finite Borel measure $\mu$ supported on a Lebesgue measure zero set in $\rd$, the Fourier restriction problem asks  for which values of $q'\in[1,2)$, $p'\in[1,\infty]$ the Fourier transform is $L^{q'}(\rd) \to L^{p'}(\mu)$ bounded. That is, when does the restriction estimate
\begin{equation}\label{eq:restriction}
  \| \widehat{f}\, \|_{L^{p'}(\mu)} \lesssim \| f \|_{{L^{q'}(\rd)}}
\end{equation}
hold for all $f\in L^{q'}(\rd)$. By $L^p$ duality, this is equivalent to asking for the Fourier extension estimate
\begin{equation}\label{eq:extension}
    \| \widehat{f\mu} \|_{L^{q}(\rd)} \lesssim \| f \|_{L^{p}(\mu)}
\end{equation}
to hold for all $f\in L^{p}(\mu)$. Where we use the notation $A\lesssim B$ to mean that there exists a uniform constant $c$ such that $A\leq cB$. We also write  $A\approx B$ if $A\lesssim B$ and $B\lesssim A$; and $p'$ and $q'$ for the H\"older conjugate of $p,q \in [1,\infty]$, defined by $\frac{1}{p} + \frac{1}{p'} = \frac{1}{q} + \frac{1}{q'}=1$.

As mentioned above, particular instances of this problem having become central questions in the field. For example, in the case of the sphere and paraboloid it is conjectured that \eqref{eq:restriction} holds for the surface measure if
\begin{equation}\label{eq:conjecturepq}
    \frac{d-1}{p'}\geq \frac{d+1}{q}\quad \text{and}\quad q>\frac{2d}{d-1}.
\end{equation}

The Fourier restriction conjecture lies at the centre of many deep problems in harmonic analysis, and is partly motivated by its applications to PDEs. For example, the extension operator on the paraboloid $\widehat{f\sigma_{\mathbb{P}}}$ is the solution to the free Schr\"odinger equation, which describes the evolution of the wave function of a free particle in quantum mechanics
\begin{equation}\label{eq:Schrodinger}
    \begin{cases}
      \frac{\partial u}{\partial t}(x,t) = 2\pi i\,\Delta u(x,t), &(x,t)\in \R^{d-1}\times \R;\\
      ~~~ u(x,0) = g(x), &x\in\R^{d-1}.
    \end{cases}
\end{equation}
Here  $\sigma_{\mathbb{P}}$ is the surface measure on the paraboloid and $f = \widecheck{g}$. If the Fourier extension estimate  \eqref{eq:extension} holds for the paraboloid and $\widecheck{g}\in L^p(\sigma_{\mathbb{P}})$ for some $p\in[1,\infty]$, then $u\in L^q(\rd)$ where $p,q\in[1,\infty]$ satisfy \eqref{eq:conjecturepq}.


 Brownian motion is one of the most important, widely studied, and well understood   stochastic processes. It is defined as the unique (random) map $W:[0,1]\to\R$ with independent increments, that is almost surely continuous, satisfying $W(0)=0$ almost surely, and such that for $t,s\in\R$ $t>s$, $W(t)-W(s)\sim N(0,t-s)$. Its fractal properties have been studied  since the work of Levy \cite{Lev53} and Taylor \cite{Tay53}, and its Fourier analytic properties since at least Kahane \cite{kahane1966a, kahane1966b}.
 
 One may generalise Brownian motion to a map $W:[0,1]^k\to\R^l$, which is often referred to as the Brownian sheet or the $(k,l)$-Brownian sheet. In this paper we focus on the case $l=1$ and we will work with a simpler version of the $(k,1)$-Brownian sheet known as the additive Brownian sheet. This reduction makes the process easier to study, while still resembling the standard $(k,1)$-Brownian sheet (see \cite{KY02} for a more detailed discussion on additive processes). We define the  \emph{$(k,1)$-additive Brownian sheet} as the map $W:[0,1]^k\to\R$ satisfying that for $t\in[0,1]^k$, $t=(t_{1},\ldots,t_{k})$,
\begin{equation*}
    W(t) = \sum_{i=1}^k W^i(t_{i}),
\end{equation*}
where $W^i:[0,1]\to\R$ are independent $1$-parameter Brownian motions. This process is usually considered as a good proxy for the non-additive case and many properties which hold in the additive case can be shown to hold also in the non-additive case.  

Our main object of study is the graph of the $(k,1)$-additive Brownian sheet defined as
\begin{equation*}
    G(W) = \{ (t,W(t)) : t\in[0,1]^k\},
\end{equation*}
and we will refer to $G(W)$  as  the \textit{additive Brownian surface} (or just \textit{Brownian surface} in the non-additive case). The additive Brownian surface is a simply connected manifold of topological co-dimension 1 living in $\mathbb{R}^{k+1}$.  It comes with a natural `surface measure' which is simply the lift of Lebesgue measure on $[0,1]^k$ to the graph, that is, we define the \emph{surface measure} $\mu$ on the additive Brownian sheet for Borel $E\subset\R^{k+1}$ by
\begin{equation*}
    \mu(E) = \mathcal{L}^k\{ t\in[0,1]^k : (t,W(t))\in E \},
\end{equation*}
where $\mathcal{L}^k$ is the $k$-dimensional Lebesgue measure.  Our main question is now the following.

\begin{ques}
For which values of $p,q$ does the extension estimate \eqref{eq:extension} hold almost surely for the surface measure on the Brownian sheet? 
\end{ques}

We make partial progress towards answering this question in this paper. One may also consider the following less precise questions.

\begin{ques}
For which values of $p,q$ does the extension estimate \eqref{eq:extension} hold almost surely for natural measures supported fractal manifolds generated by important stochastic processes, such as the (non-additive) $(k,l)$-Brownian sheet, analogous fractional Brownian sheets, Levy processes, the Gaussian free field, etc?   
\end{ques}
Although these questions are of independent interest, it might be interesting to further investigate the connection to (stochastic) PDEs. Following the same steps as for the Schr\"odinger equation \eqref{eq:Schrodinger} we can define the random operator $T$ to be the Fourier multiplier with symbol $W(\xi)$. That is, $T$ is the operator such that $\widehat{Tu}(\xi,t) = W(\xi)\widehat{u}(\xi,t)$, or $Tu(x,t) = (\widecheck{W}*u)(x,t)$ (where $*$ is the convolution operator in the $x$ variable). Then $u(x,t)=\widehat{f\mu}(x,t)$ is the solution to the equation
\begin{equation}\label{eq:randomPDE}
    \begin{cases}
      \frac{\partial u}{\partial t}(x,t) = 2\pi i\,Tu(x,t), &(x,t)\in [0,1]^{k}\times \R;\\
      ~~~u(x,0) = f(x), &x\in[0,1]^{k},
    \end{cases}
\end{equation}
with $f = \widecheck{g}$ and $\mu$ the surface measure on the additive Brownian surface. The Fourier restriction estimate we obtain in Theorem~\ref{thm:restrictionBM} below implies that whenever $\widehat{f}\in L^2(\mu)$, almost surely the solution of \eqref{eq:randomPDE} is $u\in L^q(\rd)$ for any $q>4$ if $k=1$, or $q>\frac{8k+2}{3k}$ if $k\geq2$.  It would be interesting to know if there are useful physical interpretations of these equations or if there are connections to existing PDEs in the literature, but we do not pursue this here.

\begin{figure}[H]
  \centering
  \begin{subfigure}{0.5\textwidth}
    \includegraphics[scale=0.4]{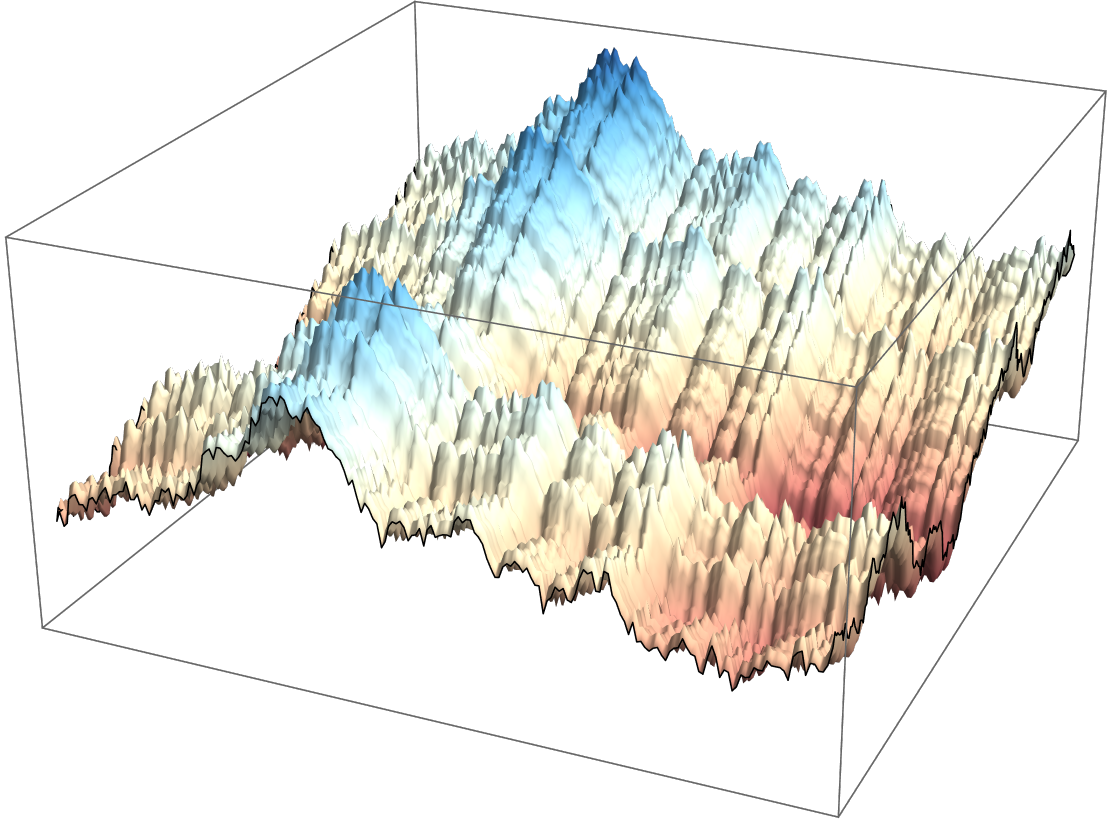}  
  \end{subfigure}%
  \begin{subfigure}{0.5\textwidth}
    \includegraphics[scale=0.4]{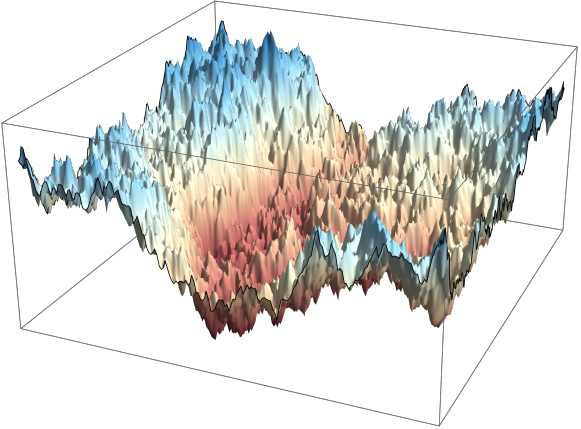}
  \end{subfigure}
  \caption{Two realisations of the Brownian surface for $k=2$. Left: in the additive case. Right: in the non-additive case.}
\end{figure}

\subsection{Dimension theory for Brownian surfaces and Stein--Tomas}

The dimension theory of the graph of standard Brownian motion was first studied in \cite{Tay53}, where it was proved that $\hd G(W) = \frac{3}{2}$ almost surely. This result was generalised by \cite{Yod75} to the (non-additive) Brownian sheet, showing that $\hd G(W)= k + 1/2$ and this formula also holds in the additive case.  The Fourier dimension of the graph of standard Brownian motion was shown to be equal to 1 in \cite{FOS14,FS18}.  The dimensional  results mentioned here also hold with the graph replaced by the surface measure, that is, the dimension of the surface is \emph{witnessed} by the surface measure.  We show below that this also holds for the Fourier dimension of the additive Brownian surface  when $k=2$ but for $k \geq 3$ one should not expect this to hold. Indeed, one might conjecture that the Fourier dimension of the surface is $k$ but we show that the Fourier dimension of the surface measure is $2$ for all $k \geq 2$. 

One reason why these Brownian graphs and surfaces are interesting examples in Fourier restriction theory is that they are \emph{not} Salem sets (they have distinct Hausdorff and Fourier dimensions).  This   was perhaps a surprise since Kahane \cite{Ka1} established that many other stochastically defined sets are Salem, including Brownian images. Indeed, it is often expected for random measures  to exhibit Fourier decay. A classic example of such behaviour can be seen in constructions of random Cantor measures \cite{SS17,SS18}. 

The Stein--Tomas theorem gives sufficient conditions for the extension  estimate \eqref{eq:extension} to hold in the case $p=2$ which depend on the Fourier and Frostman dimensions of the measure $\mu$. It states that the estimate
  \begin{equation}\label{eq:SteinTomas}
    \| \widehat{f\mu} \|_{L^{q}(\rd)} \lesssim \| f \|_{L^{2}(\mu)}
\end{equation}
holds for all $f\in L^2(\mu)$ as long as $q> 2 + 4 \frac{d-\frd\mu}{\fd\mu}$, where $\fd\mu$ is the Fourier dimension (see Section~\ref{sec:dimensions}), and $\frd\mu$ is the Frostman dimension
\begin{equation*}
    \frd\mu = \sup\{ \alpha\in[0,d]: \mu\big( B(x,r) \big)\lesssim r^\alpha,\,x\in\rd,\,r>0 \}.
\end{equation*}
As such, we can already get some information for the surface measure on the graph of standard Brownian motion given that the Fourier dimension is 1 and the Frostman  dimension is 3/2. We obtain that the estimate \eqref{eq:SteinTomas} holds for all $f\in L^{2}(\mu)$ whenever $q>4$.

\subsection{The Fourier spectrum and improved restriction estimates}\label{sec:dimensions}

The connection between Fourier decay and the potential theoretic method for   Hausdorff dimension gives  rise to the  Fourier spectrum, originally defined  in \cite{Fra24}.  This is a   family of dimensions that continuously interpolate between the the Fourier and Hausdorff dimensions for sets and the Fourier and Sobolev dimensions for measures. We now  recall the key definitions, referring the reader to \cite{Fra24, mattilaFourier} for more background. Recall first that the Fourier dimension of a Borel measure $\mu$ is defined by
\begin{equation*}
  \fd \mu = \sup\Big\{ s>0:   \sup_{\xi\in\rd}\big| \widehat{\mu}(\xi) \big|^2|\xi|^{s}<\infty \Big\}
\end{equation*}
where $\widehat{\mu}$ denotes the Fourier transform of $\mu$.  In particular, the Fourier dimension quantifies uniform decay of the Fourier transform. The Fourier dimension of a Borel set  $X$ is then
\begin{equation*}
  \fd X = \sup\Big\{ s\in[0,d] : \exists \mu\text{ on }X : \sup_{\xi\in\rd}\big| \widehat{\mu}(\xi) \big|^2|\xi|^{s}<\infty \Big\}.
\end{equation*}
The Sobolev dimension of $\mu$ is defined by
\begin{equation*}
    \sd \mu  = \sup\bigg\{ s>0 : \int_{\rd} \big| \widehat{\mu}(\xi) \big|^{2}|\xi|^{s - d} \,d\xi < \infty \bigg\}.
\end{equation*}
That is, the Sobolev dimension quantifies average Fourier decay. The Hausdorff dimension of $X$ may also be expressed in terms of average Fourier decay as follows. Indeed, 
\begin{equation*}
    \hd X = \sup\bigg\{ s\in[0,d] :\exists \mu\text{ on }X: \int_{\rd} \big| \widehat{\mu}(\xi) \big|^{2}|\xi|^{s - d} \,d\xi < \infty \bigg\}.
\end{equation*}
It is easy to see that $\fd X \leq \hd X$ and that $\fd \mu \leq \sd \mu$. The Fourier spectrum is a family of dimension that continuously interpolate between these isolated notions of dimension. Define the $(s,\theta)$-energies for $\theta\in(0,1]$ as
\begin{equation*}
    \J_{s,\theta} (\mu) = \bigg( \int_{\rd} \big| \widehat{\mu}(\xi) \big|^{\frac{2}{\theta}} |\xi|^{\frac{s}{\theta}-d} \,d\xi \bigg)^\theta,
\end{equation*}
and for $\theta = 0$ as
\begin{equation*}
    \J_{s,0}(\mu) = \sup_{\xi\in\rd}\big| \widehat{\mu}(\xi) \big|^2|\xi|^{s}.
\end{equation*}
Then the Fourier spectrum of $\mu$ at $\theta\in[0,1]$ is given by
\begin{equation*}
    \fs \mu = \sup\{ s>0 : \J_{s,\theta}(\mu)<\infty \},
\end{equation*}
and the Fourier spectrum of $X$ at $\theta\in[0,1]$ is
\begin{equation*}
    \fs X = \sup\{ s\in[0,d] :\exists \mu\text{ on }X: \J_{s,\theta}(\mu) < \infty \}.
\end{equation*}
The Fourier spectrum satisfies   $\fd X = \fd^0 X \leq \fs X\leq \fd^1 X = \hd X$, and   is non-decreasing and continuous for $\theta\in[0,1]$.  For measures, the Fourier spectrum is concave, continuous on $(0,1]$ and satisfies  $\fd \mu = \fd^0 \mu \leq \fs \mu \leq \fd^1 \mu =  \sd \mu$.  Moreover, for compactly supported measures it is continuous on the closed interval $[0,1]$.

Despite its recent inception, the Fourier spectrum has already seen several applications  in both Fractal Geometry and Harmonic Analysis; see, for example, \cite{Fra24} for applications to the distance set problem,  \cite{FdO24+} for applications to exceptional set estimates for orthogonal projections, and  \cite{LL25+, LT25+} where it was used to obtain the Fourier dimension of the graph of fractional Brownian motion.  Of particular importance to us will be its application to the Fourier restriction problem, as proved in \cite{CFdO24b+}.  Here we were able to strengthen the Stein--Tomas theorem by leveraging the additional information contained in the Fourier spectrum compared with the Fourier and Sobolev dimensions considered in isolation. This generalised Stein--Tomas type theorem states that \eqref{eq:SteinTomas} holds for all $f\in L^2(\mu)$ as long as
\begin{equation*}
    q>2 + 2 \frac{(d - \frd\mu)(2-\theta)}{\fs\mu - \theta\frd\mu}
\end{equation*}
for any $\theta\in[0,1]$ such that $\fs\mu>d\theta$.  This result recovers Stein--Tomas at $\theta=0$ but by optimising over allowable $\theta$ one can often obtain stronger  estimates.  With this in mind, the strategy in this paper is to derive a precise almost sure formula for the Fourier spectrum of the surface measure on the additive Brownian surface (see Theorem \ref{thm:FSBM}) and then to use this information to derive non-trivial Fourier restriction estimates (see Theorem \ref{thm:restrictionBM}). We also consider necessary conditions for Fourier restriction.  First we obtain them directly from the Fourier spectrum (see Corollary~\ref{usingus}) and second by constructing a Knapp example (see Theorem~\ref{thm:knapp} and Theorem~\ref{generalknapp}) which also works for non-additive Brownian and fractional Brownian surfaces. Curiously we obtain the same sufficient condition from both approaches.

\section{Main Results}

Our first main result gives a precise almost sure formula for the Fourier spectrum of the surface measure on the additive Brownian surface.  The $\theta=0$ case gives the Fourier dimension and this result is new for $k \geq 2$.  For $k=1$ it was proved in \cite{FS18}.

\begin{thm}\label{thm:FSBM}
  Let $\mu$ be the surface measure on the additive Brownian surface.   Then,  almost surely, for all $\theta\in[0,1]$,
  \begin{equation*}
      \fs \mu  = \min \big\{ k + \tfrac{\theta}{2}, 2 + k\theta \big\}.
  \end{equation*}
  In particular, there is a phase transition at $\theta = \frac{k-2}{k-1/2}$, whenever $k>2$, and for $k=1,2$ it is affine.
\end{thm}

We will prove Theorem \ref{thm:FSBM} in section \ref{mainproof}. The previous theorem gives a lower bound for the Fourier spectrum of the Brownian surface itself which is sharp at $\theta=1$ and is sharp at $\theta=0$ for $k=1,2$.  For $k \geq 3$ the lower bound is not sharp and all that can be said is  the general upper bound
\[
\fs G(W) \leq \min\{ k+ (k+1)\theta, k+1/2\} 
\]
which follows by combining (the higher dimensional version of) the main result from \cite{FOS14} that any graph over $[0,1]^k$ has Fourier dimension at most $k$ (see \cite[Theorem 6.10]{mattilaFourier}), the fact that $\hd G(W) = k+1/2$, and the general upper bound for the Fourier spectrum from \cite{CFdO24a+}. To obtain the sharp result one would need to consider a different measure, most likely the lift of a Schwartz density to the graph.  We record the $\theta=0$ and $k=2$ case in the following corollary since it may be interesting in its own right.

\begin{cor}
The Fourier dimension of the additive Brownian surface $G(W)$ in the case $k=2$ is $\fd G(W) = 2$.  
\end{cor}

\begin{figure}[H]
  \begin{tikzpicture}[scale = 0.8]
      \begin{axis}[
          axis lines = left,
          xmin = 0,
          xmax = 1.05,
          ymin= 0,
          ymax = 7.2,
          ytick = {1,2,3,4,5,6,7},
          yticklabels = {$1$,$2$,$3$,$4$,$5$,$6$,$7$},
          xlabel=$\theta$,
          ylabel style={rotate=-90},
        ylabel=$\fs \mu$,
          xtick = {0.2,0.4,0.6,0.8,1},
          xticklabels = {$0.2$,$0.4$,$0.6$,$0.8$,$1$},
      ]
      \foreach \k in {1,2,3,4,5,6} {
        \addplot [
          domain=0:1, 
          thick,
          samples=100, 
      ]
      {min(2 + \k*x, \k + x/2)};
      }
      \end{axis}
  \end{tikzpicture}
  \caption{Plots of the Fourier spectrum of the surface measure $\mu$ on the additive Brownian surface $G(W)$  for $k=1,\ldots,6$; see Theorem~\ref{thm:FSBM}.}
\end{figure}
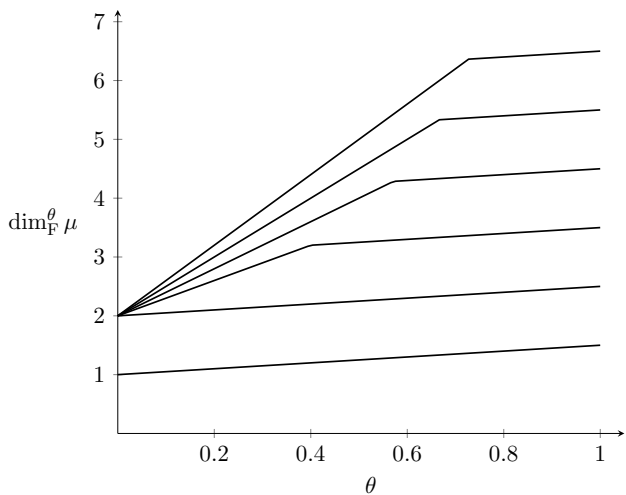

We also note the following result regarding the Frostman dimension of the surface measure on the additive Brownian surface. This result will  be needed to obtain a Fourier restriction estimate \eqref{eq:extensionST} for the additive Brownian surface measure from the Stein--Tomas theorem and its generalisation in \cite{CFdO24b+}.
\begin{prop}[Xiao]\label{cor:frostman}
  Let $\mu$ be the surface measure on the additive Brownian surface. Then, almost surely
  \begin{equation*}
      \frd\mu = k + \frac{1}{2}.
  \end{equation*}
\end{prop}
\begin{proof}
  First note that \cite[Propostion~3.2]{Xia97} still holds for the additive Brownian surface. Thus, by \cite[(3.11)]{Xia97} with $\alpha = 1/2$, there exists $C>0$ such that for any $t\in [0,1]^k$, almost surely
  \begin{equation*}
    \limsup_{r\to0} \frac{\mu\big( B((t,W(t)),r) \big)}{r^{k+\frac{1}{2}}(\log\log 1/r)^{\frac{1}{2k}}} \leq C,
  \end{equation*}
  which gives $\frd\mu\geq k + \frac{1}{2}$. This together with the fact that $\frd\mu\leq\sd\mu$  and Theorem~\ref{thm:FSBM} proves the result.
\end{proof}

Next we combine Theorem \ref{thm:FSBM} with \cite[Theorem~3.1]{CFdO24b+} to obtain the following Fourier restriction estimate for the surface measure on the additive Brownian surface. 

\begin{thm}\label{thm:restrictionBM}
  Let $\mu$ be  the surface measure on the additive Brownian surface. If $q>4$ for $k=1$, or 
  \begin{equation*}
      q>\frac{8k+2}{3k} 
  \end{equation*}
  for $k\geq2$, then almost surely, for all $f\in L^2(\mu)$,
  \begin{equation}\label{eq:extensionST}
      \| \widehat{f\mu} \|_{L^{q}(\rd)} \lesssim \| f \|_{L^2(\mu)}.
  \end{equation}
  Equivalently, if $1<q<4/3$ for $k=1$, or
  \begin{equation*}
      1\leq q'< \frac{8k+2}{5k+2},
  \end{equation*}
  for $k\geq2$, then, almost surely, for all $f\in L^{q'}(\rd)$,
  \begin{equation}\label{eq:restrictionST}
      \| \widehat{f} \|_{L^{2}(\mu)} \lesssim \| f \|_{L^{q'}(\rd)}.
  \end{equation}
\end{thm}

\begin{proof}
From  \cite[Theorem~3.1]{CFdO24b+} and Proposition~\ref{cor:frostman} we get that  for all $f\in L^2(\mu)$, \eqref{eq:extensionST} holds whenever
\[
q> 2+\frac{2-\theta}{\fs \mu - (k+1/2)\theta}
\]
for all $\theta \in [0,1]$.  We now optimise this estimate by varying $\theta$ and find that  the optimal estimate is  obtained (uniquely) at the phase transition $\theta = \frac{k-2}{k-1/2}$ for $k>2$ and (uniquely) at $\theta =0$ for $k=1,2$.
\end{proof}

Note that for $k=1,2$, the Fourier extension estimate we obtain is the same as that coming from the  Stein--Tomas theorem, that is,  $q>4,3$ respectively.  However, we obtain a better estimate than Stein--Tomas for all $k\geq3$.

Next we consider the reverse problem of finding necessary conditions for Fourier restriction.  First, we obtain the following by combining Theorem \ref{thm:FSBM} with \cite[Theorem~3.6]{CFdO24b+}.

\begin{cor} \label{usingus}
 Let $\mu$ be  the surface measure on the additive Brownian surface. Then, almost surely,  \eqref{eq:extension} cannot hold whenever
  \begin{equation*}
      q< 2 + \frac{1}{k}.
  \end{equation*}
  Equivalently, \eqref{eq:restrictionST} cannot hold whenever
  \begin{equation*}
      q'> \frac{2k+1}{k+1}.
  \end{equation*}
\end{cor}

\begin{proof}
From \cite[Theorem~3.6]{CFdO24b+}, we get that \eqref{eq:extension} cannot hold whenever 
\[
q< \sup\{2/\theta : \fs \mu < (k+1)\theta\},
\]
and combining this with Theorem \ref{thm:FSBM} yields the desired estimate.
\end{proof}

Next we derive a general necessary condition for \eqref{eq:extension} to hold by building a Knapp example.  In fact,  all we use to do this is that the additive Brownian sheet  is almost surely locally $\alpha$-H\"older for any $\alpha<1/2$.  This follows easily from the well-known result in the case $k=1$. Thus, more generally we have the following result which also holds for the non-additive Brownian sheet (with $\alpha<1/2$) and the analogous fractional Brownian surface with $\alpha< H$ where $H \in (0,1)$ is the Hurst parameter, as well as many other stochastically defined surfaces. 
\begin{thm} \label{generalknapp}
  Let $f:[0,1]^k \to\R$ be such that for some $x\in[0,1]^k$, $f$ is locally $\alpha$-H\"older   for some $\alpha>0$ at $x$. Let $\mu$ be the lift of Lebesgue measure onto the graph of $f$. Then \eqref{eq:extension} cannot hold whenever
  \begin{equation*}
      p< \frac{kq}{k(q-1)-\alpha}.
  \end{equation*}
\end{thm}

We prove Theorem \ref{generalknapp} in Section  \ref{knappproof}. Applying this to the surface measure on the additive Brownian surface we get the following.

\begin{cor}\label{thm:knapp}
  Let $\mu$ be the surface measure on the additive Brownian surface. Then, almost surely, \eqref{eq:extension} cannot hold whenever
  \begin{equation}\label{eq:knapprange}
      p< \frac{2kq}{2k(q-1)-1}.
  \end{equation}
  Equivalently, \eqref{eq:restriction} cannot hold whenever
  \begin{equation*}
      p'> \frac{2kq}{2k+1}.
  \end{equation*}
\end{cor}
Setting $p=2$ in the previous theorem (as in the Stein--Tomas estimate) gives the same estimate as in Corollary \ref{usingus}.  The exponent coming from   \cite[Theorem~3.6]{CFdO24b+} is the conjectured sharp Fourier restriction endpoint for many manifolds and this might hint that this is the right threshold here also.

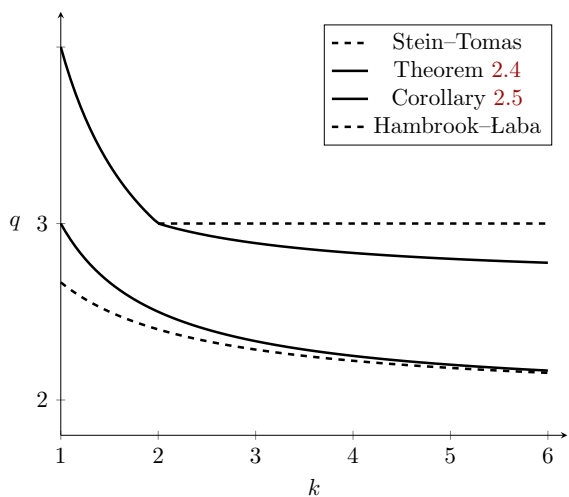
\begin{figure}[H]
    \begin{tikzpicture}[scale=0.8]
      \begin{axis}[
        axis lines = left,
        xmin = 1,
        xmax = 6.2,
        ymin= 1.8,
        ymax = 4.2,
        ytick = {2, 3,4,5,6},
        yticklabels = {$2$, $3$},
        xlabel=$k$,
        ylabel style={rotate=-90},
        ylabel=$q$,
        xtick = {1, 2, 3, 4, 5, 6, 7, 8, 9, 10},
        xticklabels = {$1$,$2$,$3$,$4$,$5$,$6$,$7$,$8$,$9$,$10$},
        legend pos=north east
    ]
   
    \addplot [ 
        domain=2:6, 
        thick,
        samples=100, 
        very thick,
        dashed
    ]
    {3};

    \addplot [ 
    domain=1:6, 
    very thick,
    samples=100,
    ]
    {max(2 + 4*(x+1-x-1/2)/x,(8*x + 2)/(3*x))};

    \addplot [ 
    domain=1:6, 
    very thick,
    samples=100, 
    ]
    {(2*x + 1)/x};

    \addplot[ 
        domain=1:6,
        very thick,
        samples=100,
        dashed,
    ]
    {(4*(x+1))/(2*x +1)};
    \legend{Stein--Tomas, Theorem~\ref{thm:restrictionBM}, Corollary~\ref{usingus},Hambrook--{\L}aba}

    \end{axis}
    \end{tikzpicture}
    \caption{Bounds for the range of $q$ for the Fourier extension estimate to hold and not hold for the Brownian sheet; see Theorem~\ref{thm:restrictionBM} and Corollary~\ref{usingus}. By ``Hambrook--{\L}aba'' we refer to the observation made in \cite{HL13} that no extension estimate will hold for $2\leq q < \frac{2d}{\sd \mu} = \frac{2(k+1)}{k + 1/2}$. These plots should be   understood as only applying to integer points in the domain, but we included the
full curves for aesthetic reasons.  In particular, using the Fourier spectrum bounds the sharp threshold between the solid curves and appealing to previous estimates bounds the sharp threshold between the dashed curves.}\label{fig:restrictionBM}
  \end{figure}

\section{Proof of Theorem~\ref{thm:FSBM}} \label{mainproof}

\subsection{Decomposition of the Fourier transform and almost sure decay estimates}

Let $W:[0,1]^k\to\R$ be the additive Brownian sheet,
\begin{equation*}
    W(t) = \sum_{i=1}^k W^i(t_{i}),
\end{equation*}
where $W^i:[0,1]\to\R$ are independent $1$-parameter Brownian motions. Let $\mu$ be the associated surface measure, that is, the push-forward of Lebesgue measure on $[0,1]^k$ to $W$ defined for $E\subset\R^{k+1}$ by
\begin{equation*}
    \mu(E) = \mathcal{L}\{ t\in[0,1]^k : (t,W(t))\in E \}.
\end{equation*}
Given $\xi\in\R^{k+1}$ with $\xi = (\xi_{1},\ldots, \xi_{k},y)$, the Fourier transform decomposes as
\begin{align*}
    \widehat{\mu}(\xi) &= \int_{[0,1]^k} e^{-2\pi i(\xi_{1}t_{1}+\cdots+\xi_{k}t_{k} + yW(t))} \,dt\\
    &= \int_{0}^1 e^{-2\pi i(\xi_{1}t_{1} + yW^1(t_{1}))} \,dt_{1}\cdots\int_{0}^1 e^{-2\pi i (\xi_{k}t_{k} + yW^k(t_{k}))}\, dt_{k}\\
    &= \widehat{\mu_{1}}(\xi_{1},y)\cdots\widehat{\mu_{k}}(\xi_{k},y),
\end{align*}
where for each $j=1,\ldots, k$, $\mu_{j}$ is the surface measure on the graph of $(1,1)$-Brownian motion. That is, each $\mu_{j}$ is the push-forward of Lebesgue measure on $[0,1]$ to the graph of the Brownian motion $W^i$.  Fortunately, we already have good almost sure estimates for the Fourier transforms of these measures from \cite{FS18} and we recall these estimates below.

In \cite{FS18} the authors estimate $\widehat{\mu_{j}}(\xi_{j},y)$ by splitting the frequency space into regions depending on the angle that $(\xi_{j},y)$ makes with the origin. We will now use those estimates, in a slightly more explicit form. For the `horizontal angles' (that is, the case when $|\xi_j|\geq |y|^2$), \cite[proof of Lemma~3.4]{FS18}, and some random $T>0$,
\begin{equation*}
    \big|\widehat{\mu_{j}}(\xi_{j},y)\big|\leq \bigg| \int_{T}^1 e^{-2\pi i (\xi_{j}t_{j} + yW(t_{j}))} \,dt_{j}\bigg| \lesssim |\xi_{j}|^{-1},
\end{equation*}
almost surely.  Also, \cite[proof of Lemma~3.6 and Lemma~3.5]{FS18} gives for the same    $T>0$,
\begin{equation*}
    \big|\widehat{\mu_{j}}(\xi_{j},y)\big|\leq \bigg| \int_{0}^T e^{-2\pi i (\xi_{j}t_{j} + yW(t_{j}))} \,dt_{j}\bigg| \lesssim |y||\xi_{j}|^{-1},
\end{equation*}
almost surely. This estimate comes from It\^o calculus. Finally, the third estimate comes from the `vertical angles' (that is, the case when   $|\xi_j| \leq  |y|^2$).  In this case,  \cite[proof of Lemma~3.10 and Lemma~3.7]{FS18} gives that almost surely
\begin{equation*}
    \big| \widehat{\mu_{j}}(\xi_{j},y) \big| \lesssim |y|^{-1} \sqrt{\log |y|}.
\end{equation*}
This estimate   uses ideas from Kahane's work on Brownian images. Therefore, putting these estimates together we obtain that almost surely
\begin{align*}
   \big| \widehat{\mu_{j}}(\xi_{j},y) \big| &\lesssim \min\Big\{ \max \big\{  |y||\xi_{j}|^{-1}, |\xi_{j}|^{-1} \big\}, |y|^{-1}\sqrt{\log |y|}, 1 \Big\}\\
    &=\begin{cases}
      |y||\xi_{j}|^{-1}\,, &\text{if }1\leq|y|\leq|y|^2\leq|\xi_{j}|;\\
      |\xi_{j}|^{-1}\,, &\text{if }|y|\leq1\leq|\xi_{j}|;\\
      |y|^{-1} \sqrt{\log |y|} \,, &\text{if } \max\{|\xi_{j}|,1\}\leq|y|^{2};\\
      1\,,&\text{if } \max\{|\xi_{j}|,|y|\} \leq1.
    \end{cases}\numberthis\label{eq:minmax}
\end{align*}
The implicit constants in the above are random, but that is sufficient for what follows.  The almost sure bounds in \eqref{eq:minmax} are sufficient to prove the almost sure uniform bound
\[
 \big| \widehat{\mu_{j}}(\xi) \big| \lesssim |\xi|^{-1/2} \sqrt{\log |\xi|} \qquad (\xi \in \mathbb{R}^2)
\]
which gave the main result of \cite{FS18} that the Fourier dimension of the graph of Brownian motion and the surface measure on this graph both have Fourier dimension 1.  However, one can see from \eqref{eq:minmax}  that this `worst case' decay rate happens relatively rarely and that a much better decay
\[
 \big| \widehat{\mu_{j}}(\xi) \big| \lesssim |\xi|^{-1 }   \qquad (\xi \in \mathbb{R}^2)
\]
happens most of the time.  The Fourier spectrum leverages this and quantifies how good the average Fourier decay is.

\subsection{Proof of the lower bound} 

Fix $\theta \in (0,1]$. We may assume that $k \geq 2$ since the case $k=1$ is dealt with by the result of \cite{FS18} and the fact that the Fourier spectrum of a measure is concave (and so the affine lower bound  comes for free).    We need to prove that the energies $ \J_{s,\theta}(\mu)$ are finite for appropriate values of $s$.  To achieve this, we split the energy integral into 6 regions defined below.  By symmetry, we may restrict to the case when all the coordinates of $\xi$ are positive and assume that $0<\xi_{k}<\xi_{k-1}<\cdots<\xi_{2}<\xi_{1}$.  Moreover, for convenience,  we assume without loss of generality that   $y$ and $y^2$ are distinct from all the $\xi_j$ and from 1.  The alternative occurs on a measure zero set and therefore can be omitted form the integral.  Further, we assume that at least one of the coordinates of $\xi=(\xi_1, \dots, \xi_k, y)$ is strictly greater than $1$, which can be done since the integral is always finite when restricted to a neighbourhood of the origin. This means we assume $\max\{\xi_1, y\} >1$.  Therefore, in the third case of \eqref{eq:minmax} (i.e. when $\max\{|\xi_{j}|,1\}<y^{2}$) it could be that $1<y<\xi_{j}<y^2$, $1< \xi_{j}<y<y^2$, or $\xi_{j}<1<y<y^2$.

With this, we only need to work out the different cases for each possible position of $1$, $y$, and $y^2$ relative to the list of $\xi_j$. That is, for $j_{1},j_{2},j_{3}\in\{ 0,\ldots,k+1 \}$ we consider the regions given by
\begin{equation*}
    0<\xi_{k}<\cdots<\xi_{j_{3}} <1<\xi_{j_{3}-1} < \cdots < \xi_{j_{2}} < y < \xi_{j_{2}-1} < \cdots< \xi_{j_{1}} < y^2 <\xi_{j_{1}-1}  < \cdots<\xi_{1},
\end{equation*}
and
\begin{equation*}
    0<\xi_{k}<\cdots<\xi_{j_{3}} <y^2<\xi_{j_{3}-1} < \cdots < \xi_{j_{2}} < y < \xi_{j_{2}-1} < \cdots< \xi_{j_{1}} < 1 <\xi_{j_{1}-1}  < \cdots<\xi_{1}.
\end{equation*}
Two cases are needed here to account for the two possible orderings $1<y<y^2$ and $y^2<y<1$.   We are left with a lot of cases to consider, each taking the form of an energy integral restricted to a certain region.  For each integral, we need to determine conditions on $s$ which ensure finiteness.  Then the lower bound for the Fourier spectrum will be the minimum of all of these bounds, that is, we require all of the integrals to be finite simultaneously.    Fortunately, we can dramatically reduce the number of cases via a linearity argument.  Indeed, it is straightforward to see (and will become apparent below) that the bounds obtained for the Fourier spectrum in each case  are  \emph{linear} in the variables $j_{1},j_{2}$, and $j_{3}$. Therefore, since linear functions on   an interval achieve their extrema on the boundary,  the dominant  bounds will be obtained in  the following boundary cases:

For $y>1$:
\begin{enumerate}[leftmargin=*,label=\textbf{Case \arabic*.}]
  \item $j_{1} = 0,$ $j_{2} = 0$, $j_{3} = 1$:
  \begin{equation*}
       0<\xi_{k}<\cdots <\xi_{1}<1<y<y^2.
  \end{equation*}
  \item $j_{1} = 0$, $j_{2} = 1$, $j_{3} = k+1$:
  \begin{equation*}
       1<\xi_{k}<\cdots <\xi_{1}<y<y^2.
  \end{equation*}
  \item $j_{1} = 1$, $j_{2} = k+1$, $j_{3} = k+1$:
  \begin{equation*}
       1<y<\xi_{k}<\cdots <\xi_{1}<y^2.
  \end{equation*}
  \item $j_{1} = k+1$, $j_{2} = k+1$, $j_{3} = k+1$:
  \begin{equation*}
      1<y<y^2<\xi_{k}<\cdots <\xi_{1}.
  \end{equation*}
\end{enumerate}
For the cases where $y<1$ recall that we   assume that $\xi_{1}>1$. Therefore, we consider only:
\begin{enumerate}[leftmargin=*,label=\textbf{Case \arabic*.}]
\addtocounter{enumi}{4}
  \item $j_{1} = 2,$ $j_{2} = 2$, $j_{3} = 2$:
  \begin{equation*}
       0<\xi_{k}<\cdots <\xi_{2}<y^2<y<1<\xi_{1}.
  \end{equation*}
  \item $j_{1} = k+1$, $j_{2} = k+1$, $j_{3} = k+1$:
  \begin{equation*}
      y^2<y<1<\xi_{k}<\cdots <\xi_{1}.
  \end{equation*}
\end{enumerate}

 We begin with the lower bound. We will split the $(s,\theta)$-energy as follows
 \begin{equation*}
    \J_{s,\theta}(\mu)^{\frac{1}{\theta}} \approx  \int_{|\xi|>1} \big| \widehat{\mu}(\xi) \big|^{\frac{2}{\theta}} |\xi|^{\frac{s}{\theta}-k-1} \,d\xi \approx \sum_{\ell=1}^6 \int_{\xi\text{ in Case }\ell} \big| \widehat{\mu}(\xi) \big|^2 |\xi|^{\frac{s}{\theta}-k-1} \,d\xi \eqqcolon \sum_{\ell=1}^6 J_{s,\theta}^\ell(\mu),
\end{equation*}
 and prove that for $s<\min \big\{ k + \tfrac{\theta}{2}, 2 + k\theta \big\}$, all of the integrals  $J_{s,\theta}^1(\mu)\ldots J_{s,\theta}^6(\mu)$ are finite, where  $J_{s,\theta}^\ell(\mu)$ refers to the energy integral restricted to the region defined above in Case $\ell$.

\subsubsection*{Case 1: $0<\xi_{k}<\cdots <\xi_{1}<1<y<y^2$}

We will use that for $j=1,\ldots,k$, $\big|\widehat{\mu_{j}}(\xi_{j},y)\big|\lesssim y^{-1}\sqrt{\log y}$ almost surely, and $|\xi|\approx y$. Then, almost surely,
\begin{align*}
    J_{s,\theta}^{1}(\mu) 
    &\lesssim \int_{y=1}^\infty y^{\frac{s}{\theta}-k-1} \int_{\xi_{1}=0}^1 y^{-\frac{2}{\theta}}(\log y)^{\frac{1}{\theta}} \int_{\xi_{2}=0}^{\xi_{1}} y^{-\frac{2}{\theta}}(\log y)^{\frac{1}{\theta}}  \cdots \int_{\xi_{k}=0}^{\xi_{k-1}} y^{-\frac{2}{\theta}}(\log y)^{\frac{1}{\theta}} \,d\xi_{k}\cdots d\xi_{2}\,d\xi_{1}\,dy \\
    &\lesssim \int_{y=1}^\infty y^{\frac{s-2k}{\theta}-k-1} (\log y)^{\frac{k}{\theta}} \,dy<\infty,
\end{align*}
provided $s < 2k+k\theta$.

\subsubsection*{Case 2: $1<\xi_{k}<\cdots<\xi_{1}<y<y^2$}

For $j=1,\ldots,k$, $\big|\widehat{\mu_{j}}(\xi_{j},y)\big|\lesssim y^{-1}\sqrt{\log y}$ almost surely, and $|\xi| \approx y$. Thus, almost surely, 
\begin{align*}
     J_{s,\theta}^{2}(\mu) &\lesssim\int_{y=1}^\infty y^{\frac{s}{\theta}-k-1} \int_{\xi_{1}=1}^{y} y^{-\frac{2}{\theta}} (\log y)^{\frac{1}{\theta}} \int_{\xi_{2}=1}^{\xi_{1}} y^{-\frac{2}{\theta}} (\log y)^{\frac{1}{\theta}} \cdots \int_{\xi_{k}=1}^{\xi_{k-1}} y^{-\frac{2}{\theta}}(\log y)^{\frac{1}{\theta}}  \,d\xi_{k}\cdots d\xi_{2}\,d\xi_{1}\,dy \\
     &\approx \int_{y=1}^\infty y^{\frac{s-2k}{\theta}-1}(\log y)^{\frac{k}{\theta}}  \,dy< \infty,
\end{align*}
provided $s<2k$.

\subsubsection*{Case 3: $1<y<\xi_{k}<\cdots <\xi_{1}<y^2$} In this case, for $j=1,\ldots,k$, $\big|\widehat{\mu_{j}}(\xi_{j},y)\big|\lesssim y^{-1}\sqrt{\log y} $ almost surely, and $|\xi|\approx \xi_{1}$. We may safely assume that $s>\theta$ since $J_{\theta,\theta}(\mu)<\infty$ by concavity of the Fourier spectrum for measures.  Then, almost surely, 
\begin{align*}
  J_{s,\theta}^3(\mu) &\lesssim  \int_{y=1}^{\infty} y^{-\frac{2k}{\theta}} (\log y)^{\frac{k}{\theta}}  \int_{\xi_{1}=y}^{y^2} \xi_{1}^{\frac{s}{\theta}-k-1}\int_{\xi_{2}=y}^{\xi_{1}} \cdots \int_{\xi_{k}=y}^{\xi_{k-1}}\,d\xi_{k}\cdots\,d\xi_{2}\,d\xi_{1} \,dy\\
  &\approx   \int_{y=1}^{\infty} y^{-\frac{2k}{\theta}} (\log y)^{\frac{k}{\theta}}  \int_{\xi_{1}=y}^{y^2} \xi_{1}^{\frac{s}{\theta}-2} \,d\xi_{1} \,dy\\
    &\lesssim   \int_{y=1}^{\infty} y^{\frac{2s-2k}{\theta}-2} (\log y)^{\frac{k}{\theta}}  \,dy<\infty, 
\end{align*}
provided $s<k+\frac{\theta}{2}$.

\subsubsection*{Case 4: $1<y<y^2 < \xi_{k}< \cdots<\xi_{1}$}

Now for $j=1,\ldots,k$, $\big|\widehat{\mu_{j}}(\xi_{j},y)\big|\lesssim y\xi_{j}^{-1}$ almost surely, and $|\xi|\approx \xi_{1}$. Then, almost surely, 
\begin{align*}
    J_{s,\theta}^{4}(\mu) &\lesssim \int_{\xi_{1}=1}^{\infty} \xi_{1}^{\frac{s-2}{\theta}-k-1} \int_{\xi_{2}=1}^{\xi_{1}} \xi_{2}^{-\frac{2}{\theta}} \cdots \int_{\xi_{k}=1}^{\xi_{k-1}} \xi_{k}^{-\frac{2}{\theta}} \int_{y=1}^{\xi_{k}^{1/2}} y^{\frac{2k}{\theta}} \,dy \,d\xi_{k}\cdots\,d\xi_{2}\,d\xi_{1} \\
    &\approx \int_{\xi_{1}=1}^{\infty} \xi_{1}^{\frac{s-2}{\theta}-k-1} \int_{\xi_{2}=1}^{\xi_{1}} \xi_{2}^{-\frac{2}{\theta}} \cdots\int_{\xi_{k}=1}^{\xi_{k-1}} \xi_{k}^{-\frac{2}{\theta} + \frac{k}{\theta} + \frac{1}{2}} \,d\xi_{k}\cdots\,d\xi_{2}\,d\xi_{1} \\
    &\approx \int_{\xi_{1}=1}^{\infty} \xi_{1}^{\frac{s-2}{\theta}-k-1 + \big( 1-\frac{2}{\theta} \big)(k-1) + \frac{k}{\theta} + \frac{1}{2}}\,d\xi_{1}\\
     &= \int_{\xi_{1}=1}^{\infty} \xi_{1}^{\frac{s-k}{\theta}  - \frac{3}{2}}\,d\xi_{1} < \infty, 
\end{align*}
provided $s<k+\frac{\theta}{2}$.

\subsubsection*{Case 5: $0<\xi_{k}<\cdots <\xi_{2}<y^2<y<1<\xi_{1}$}\label{case5}
In this case we will use the bounds $\big|\widehat{\mu_{j}}(\xi_{j},y)\big| \lesssim 1$ for $i=2,\ldots,k$ and $\big|\widehat{\mu_{1}}(\xi_{1},y)\big| \lesssim |\xi|^{-1}$ almost surely, and $|\xi|\approx\xi_{1}$. This gives, almost surely, 
\begin{equation*}
    J_{s,\theta}^{5}(\mu) \lesssim \int_{\xi_{1}=1}^{\infty} \xi_{1}^{\frac{s-2}{\theta}-k-1} \int_{y=0}^{1} \int_{\xi_{2}=0}^{y^2}\cdots\int_{\xi_{k}=0}^{\xi_{k-1}} d\xi_{k}\cdots d\xi_{2}\,dy\,d\xi_{1} = \int_{\xi_{1}=1}^\infty \xi_{1}^{\frac{s-2}{\theta}-k-1} < \infty,
\end{equation*}
provided $s<2 + k\theta$.

\subsubsection*{Case 6: $0<y^2<y<1<\xi_{k}<\cdots <\xi_{1}$}

For $j=1,\ldots,k$, $\big|\widehat{\mu_{j}}(\xi_{j},y)\big|\lesssim |\xi_j|^{-1}$ almost surely, and $|\xi|\approx \xi_{1}$. This yields, almost surely, 
\begin{align*}
  J_{s,\theta}^{6}(\mu) &\lesssim \int_{\xi_{1}=1}^{\infty}\xi_{1}^{\frac{s-2}{\theta} -k-1} \int_{\xi_{2}=1}^{\xi_{1}}\xi_{2}^{-\frac{2}{\theta}} \cdots \int_{\xi_{k}=1}^{\xi_{k-1}} \xi_{k}^{-\frac{2}{\theta}} \int_{y=0}^1  \,dy\,d\xi_{k}\cdots\,d\xi_{2}\,d\xi_{1} \\
  &\approx\int_{\xi_{1}=1}^{\infty}\xi_{1}^{\frac{s-2}{\theta} -k-1}\,d\xi_{1} <\infty,
\end{align*}
provided $s < 2+k\theta$.

The minimum bounds for $s$ are obtained in Cases 3, 4, 5 and 6. Therefore, almost surely, 
\begin{equation*}
    \fs \mu\geq\min \big\{ k + \tfrac{\theta}{2}, 2 + k\theta \big\},
\end{equation*}
which completes the proof of the lower bound.

\subsection{Proof of the upper bound} 

First note that the projection of $\mu$ onto the first $k$ coordinates is the $k$-dimensional Lebesgue measure $\mathcal{L}^{k}$. Therefore, by \cite[Proposition~4.2]{FdO24+} and \cite[Proposition~6.1]{Fra24+}
\begin{equation}\label{eq:UpperProj}
    \fs \mu \leq \fs \mathcal{L}^{k} + \theta  = \big(2 + (k-1)\theta\big) + \theta = 2 + k\theta,
\end{equation}
which proves the first upper bound.


We now turn to the second upper bound.  \emph{A priori} this should involve seeking lower bounds for the decay rate of the Fourier transform of $\mu$, at least along certain sequences.  This information is not provided in \cite{FS18} and could be technically challenging to obtain.  Fortunately, we found an alternative approach based on establishing that the infinitude of the usual energy ($\theta =1$) for values of $s>k+1/2$  comes from a region of $\rd$ which has relatively small volume.  We then transfer this information to the energies we use via Jensen's inequality.  We believe that this approach will be useful for other problems in the future. 

 Let $k + \frac{1}{2} < s < 2 + k$, in which case all the energy integrals considered in the proof of the lower bound are finite apart from possibly $J_{s,1}^3(\mu)$ and $J_{s,1}^4(\mu)$.   This can be seen by comparing $s$ with the bounds coming from each case in the proof above. Since $\sd \mu \leq \hd G(W) = k+1/2$ we know
 \[
 \J_{s,1}(\mu) = \infty.
 \]
 Therefore, 
\begin{equation*}
  \infty = \J_{s,1}(\mu)^{\frac{1}{\theta}} \approx (J_{s,1}^1(\mu) + \cdots + J_{s,1}^{6}(\mu))^{\frac{1}{\theta}} \approx (J_{s,1}^{3}(\mu) + J_{s,1}^{4}(\mu))^{\frac{1}{\theta}}.
\end{equation*}
We will now slightly modify the regions for Cases 3 and 4 to show that in fact, the relevant estimate comes from Case 4.

Let $0<\varepsilon_{1}<\frac{3}{k+1}$ and $k+\frac{1}{2}<s<\frac{2k-1}{2-\varepsilon_{1}}+1< 2 + k$. We change the threshold from $y^2$ to $y^{2-\varepsilon_1}$ and  consider the following two new cases:
\begin{enumerate}[leftmargin=*,label=\textbf{Case \arabic*.}]
  \item[\textbf{Case 3$^*$:}]$1<y<\xi_{k}<\cdots <\xi_{1}<y^{2-\varepsilon_{1}}$.
  \item[\textbf{Case 4$^*$:}]$1<y<y^{2-\varepsilon_{1}} < \xi_{k}< \cdots<\xi_{1}$.
\end{enumerate}
We will show that for the chosen values of $s$, $J_{s,1}^{3^*}(\mu)$ is finite. From \eqref{eq:minmax} we know that for $j=1,\ldots,k$, $\big| \widehat{\mu_{j}}(\xi_{j},y) \big|\lesssim y^{-1}\sqrt{\log y}$ almost surely, and $|\xi|\approx \xi_{1}$. Therefore,
\begin{align*}
  J_{s,1}^{3^*}(\mu) &\lesssim  \int_{y=1}^{\infty} y^{-2k} (\log y)^{k}  \int_{\xi_{1}=y}^{y^{2-\varepsilon_{1}}} \xi_{1}^{s-k-1}\int_{\xi_{2}=y}^{\xi_{1}} \cdots \int_{\xi_{k}=y}^{\xi_{k-1}}\,d\xi_{k}\cdots\,d\xi_{2}\,d\xi_{1} \,dy\\
  &\approx \int_{y=1}^{\infty} y^{-2k} (\log y)^{k}  \int_{\xi_{1}=y}^{y^{2-\varepsilon_{1}}} \xi_{1}^{s-2} \,d\xi_{1} \,dy\\
    &\lesssim   \int_{y=1}^{\infty} y^{-2k + (2-\varepsilon_{1})(s-1)} (\log y)^{k}  \,dy<\infty 
\end{align*}
since $\varepsilon_{1}>0$ and $s<\frac{2k-1}{2-\varepsilon_{1}}+1$.  Since $ \J_{s,1}(\mu)=\infty$, we deduce that   $J_{s,1}^{4^*}(\mu)=\infty$.   Note that there are intermediate cases between 3* and 4* such as
\[
1<y<\xi_{k}<\cdots <\xi_{2}<y^{2-\varepsilon_{1}}<\xi_{1}
\]
and it is necessary to establish finiteness of all of these cases too in order to justify that the infinitude of $ \J_{s,1}(\mu)$ comes only from case 4*.  However, we obtain finiteness in the intermediate cases by appealing to the linearity argument used in the proof of the lower bound (or by direct calculation).

Note that for any $\varepsilon_{2}>\varepsilon_{1}$,
\begin{align*}
    \int_{\xi\text{ in Case~4}^*} |\xi|^{- \big( k + \frac{1}{2-\varepsilon_{2}} \big)} \,d\xi &\approx \int_{\xi_{1}=1}^\infty\int_{\xi_{2}=1}^{\xi_{1}}\cdots\int_{\xi_{k}=1}^{\xi_{k-1}} \int_{y=1}^{\xi_{k}^{\frac{1}{2-\varepsilon_{1}}}} \xi_{1}^{-\big( k + \frac{1}{2-\varepsilon_{2}} \big)} \,dy\,d\xi_{k}\cdots\,d\xi_{2}\,d\xi_{1} \\
    &\approx \int_{\xi_{1}=1}^{\infty}\xi_{1}^{-1+\frac{1}{2-\varepsilon_{1}} - \frac{1}{2-\varepsilon_{2}}} \,d\xi_{1}<\infty.
\end{align*}
With this, let $\varepsilon_{2}>\varepsilon_{1}$ and let $c>0$ be such that the measure $dm(\xi) = c|\xi|^{-\big(k + \frac{1}{2-\varepsilon_{2}}\big)}\,d\xi$ is a probability measure on the region given by Case 4$^*$. Then, by Jensen's inequality,
\begin{align*}
  J_{s,1}^{4^*}(\mu)^{\frac{1}{\theta}} &= \Bigg( \int_{\xi\text{ in Case~4}^*} \big| \widehat{\mu}(\xi) \big|^2 |\xi|^{s - k-1} \,d\xi\Bigg)^{\frac{1}{\theta}}\\
  &\approx \Bigg( \int_{\xi\text{ in Case~4}^*} \big| \widehat{\mu}(\xi) \big|^2 |\xi|^{s -1 + \frac{1}{2-\varepsilon_{2}}} \,dm(\xi)\Bigg)^{\frac{1}{\theta}}\\
  &\leq \int_{\xi \text{ in Case 4}^*} \big| \widehat{\mu}(\xi) \big|^\frac{2}{\theta} |\xi|^{\frac{s -1 + 1/(2-\varepsilon_{2})}{\theta}}  \,dm(\xi)\\
   &\approx \int_{\xi \text{ in Case 4}^*} \big| \widehat{\mu}(\xi) \big|^\frac{2}{\theta} |\xi|^{\frac{s -1 + 1/(2-\varepsilon_{2})}{\theta}}|\xi|^{-\big( k + \frac{1}{2-\varepsilon_{2}} \big)} \,d\xi\\
  &= \int_{\xi \text{ in Case 4}^*} \big| \widehat{\mu}(\xi) \big|^\frac{2}{\theta} |\xi|^{\frac{s -1 +\theta + (1-\theta)/(2-\varepsilon_{2})}{\theta} -k-1} \,d\xi\\
  &\leq \J_{s -1 +\theta + \frac{1-\theta}{2-\varepsilon_{2}},\theta}(\mu)^{\frac{1}{\theta}}.
\end{align*}
Since $J_{s,1}^{4^*}(\mu)=\infty$ for $s>k + \frac{1}{2}$ that implies that, $\fs\mu \leq k - \frac{1}{2}+\theta+\frac{1-\theta}{2-\varepsilon_{2}}$ letting $\varepsilon_{1},\varepsilon_{2}\to0$ yields $\fs\mu\leq k + \frac{\theta}{2}$ as we wanted to prove. Combining this with \eqref{eq:UpperProj} implies the desired almost sure upper bound
\begin{equation*}
    \fs\mu \leq \min \big\{ k + \tfrac{\theta}{2}, 2 + k\theta \big\}
\end{equation*}
for all $\theta\in[0,1]$.


\section{A Knapp example} \label{knappproof}

Recall that the Knapp example (see e.g. \cite{mattilaFourier,Dem20}) establishes a necessary condition for restriction estimates to hold for the sphere. The idea behind it is capturing a portion of the set which will yield a bad Fourier restriction estimate, which is achieved for the sphere by choosing a small `cap' which contains a large amount of measure in a relatively flat piece of the sphere.  Recall that genuinely flat sets have no Fourier decay, and therefore no Fourier restriction estimates are possible for them, and so roughly flat pieces provide natural barriers to restriction holding below certain thresholds.  When working with smooth manifolds, similar Knapp examples  also work. For the Brownian surface, we are looking for relatively flat pieces which hold a large amount of mass, but we will do this via vertical rectangles and rely only on the H\"older exponent of Brownian motion. We now construct this example and prove Theorem~\ref{generalknapp}.

\subsection{Proof of Theorem~\ref{generalknapp}}

 Fix $s\in[0,1]^k$ and $\alpha>0$ such that $f$ is locally $\alpha$-H\"older at $s$.   Then there exists $0<\delta = \delta(s,\alpha)<1$ such that for all $t\in[0,1]^k$ with $|s-t|<\delta$,
\begin{equation*}
    |f(s) - f(t)| \lesssim |s-t|^{\alpha}.
\end{equation*}
Note that $\{ (t,f(t))  : |s-t|<\delta \}$ is then contained in a box $R$ of side-lengths $\approx \delta\times \cdots\times \delta\times \delta^{\alpha}$. Let $g = 1_{R}$ and by definition the $\mu$ mass of $R$ is $\approx \delta^k$ (the Lebesgue measure of a $\delta$-cube in the domain).  Then
\begin{equation*}
    \|g\|_{L^{p}(\mu)} = \mu(R)^{\frac{1}{p}} \approx \delta^{\frac{k}{p}}.
\end{equation*}
 Let $R^*$ be the dual rectangle of $R$, that is, the rectangle centred at the origin with sides parallel to those of $R$ and of side-lengths $\approx \delta^{-1}\times\cdots\times\delta^{-1}\times\delta^{-\alpha}$. Multiplying by $|e^{2\pi i \xi\cdot (s,f(s))}|=1$ for $\xi\in\rd$,
 \begin{equation*}
  \| \widehat{g\mu} \|_{L^q(\rd)} = \Bigg( \int_{\rd} \Bigg| \int_{R} e^{-2\pi i \xi\cdot (x-(s,f(s)))} \,d\mu(x) \Bigg|^q \,d\xi \Bigg)^{\frac{1}{q}}.
 \end{equation*}
 If $\xi\in R^*/100$ and $x\in R$, then $\Re(e^{-2\pi i \xi\cdot (x-(s,f(s)))}) \approx 1$ and $\Im(e^{-2\pi i \xi\cdot (x-(s,f(s)))}) \leq 1/10$. Since the modulus of any complex number is at least its real part,
 \begin{align*}
  \| \widehat{g\mu} \|_{L^q(\rd)} &\gtrsim \Bigg( \int_{R^*/100} \Bigg| \int_{R} e^{-2\pi i \xi\cdot (x-(s,f(s)))} \,d\mu(x) \Bigg|^q \,d\xi \Bigg)^{\frac{1}{q}}\\
  &\gtrsim \Bigg( \int_{R^*/100} \big| \mu(R)  \big|^q\, d\xi \Bigg)^{\frac{1}{q}}\\
  &\gtrsim \delta^{k}|R^*|^{\frac{1}{q}}\\
  &\approx \delta^{\frac{k(q-1) - \alpha}{q}}
\end{align*}
Therefore, if we assume \eqref{eq:extension} holds, combining the estimates above yields
\begin{equation*}
    \delta^{\frac{k(q-1) - \alpha}{q}} \lesssim \| \widehat{g\mu} \|_{L^q(\rd)} \lesssim \|g\|_{L^{p}(\mu)} \lesssim \delta^{\frac{k}{p}}
\end{equation*}
and therefore
\begin{equation*}
    p\geq \frac{kq}{k(q-1)-\alpha},
\end{equation*}
as required.

\begin{figure}[H]
  \begin{center}



\includegraphics{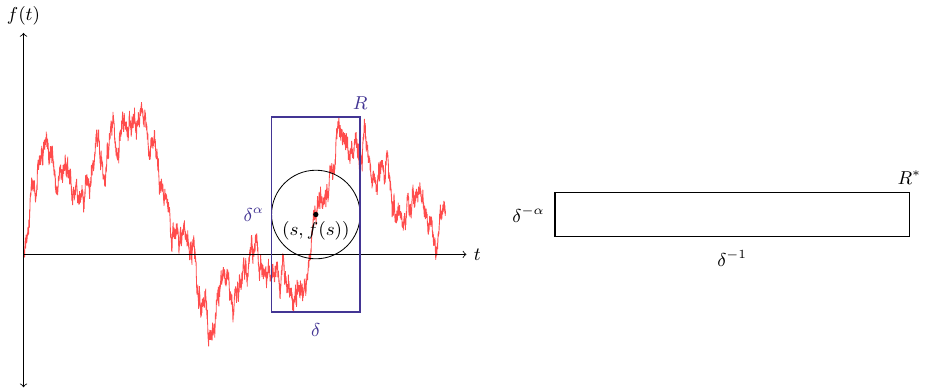}
\end{center}
\caption{Knapp example for an $\alpha$-H\"older function $f$, where $g$ is the characteristic function of the rectangle $R$ centred at $(s,f(s)) \in G(f)$ of side-lengths $\approx \delta\times\cdots\times\delta\times\delta^{\alpha}$. To the right is $R^*$ the dual rectangle of $R$.}
\end{figure}

\end{document}